\documentclass[reqno,a4paper, 11pt]{amsart}

\usepackage[utf8]{inputenc}

\usepackage[a4paper=true,pdfpagelabels]{hyperref}
\usepackage{graphicx}

\usepackage{amsfonts,epsfig}
\usepackage{latexsym}
\usepackage{mathabx}
\usepackage{amsmath}
\usepackage{amssymb}

\usepackage{color}

\newtheorem{theorem}{Theorem}
\newtheorem{lemma}[theorem]{Lemma}

\newtheorem{lettertheorem}{Theorem}
\newtheorem{letterlemma}[lettertheorem]{Lemma}

\theoremstyle{definition}

\theoremstyle{remark}

\numberwithin{equation}{section}

\newcommand{\set}[1]{\left\{#1\right\}}
\newcommand{\abs}[1]{\lvert#1\rvert}
\newcommand{\nm}[1]{\lVert#1\rVert}

\newcommand{\B}{\mathcal{B}}
\newcommand{\D}{\mathbb{D}}
\newcommand{\DD}{\widehat{\mathcal{D}}}
\newcommand{\Dd}{\widecheck{\mathcal{D}}}

\newcommand{\DDD}{\mathcal{D}}
\newcommand{\De}{\mathcal{D}}
\newcommand{\N}{\mathbb{N}}

\newcommand{\C}{\mathbb{C}}

\renewcommand{\phi}{\varphi}

\def\a{\alpha}       \def\b{\beta}        
       \def\De{{\Delta}}    
\def\la{\lambda}     \def\om{\omega}      
              
         \def\r{\rho}         \def\z{\zeta}

\def\omg{\widehat{\omega}}
\def\nug{\widehat{\nu}}

\renewcommand{\H}{\mathcal{H}}

\newenvironment{Prf}{\noindent{\emph{Proof of}}}
{\hfill$\Box$ }

\begin{document}

\title[Hilbert-type operator induced by radial weight]{Hilbert-type operator induced by radial weight}

\keywords{Hilbert operator, Hardy space, Bergman space,  reproducing kernel, radial weight. }

\author{Jos\'e \'Angel Pel\'aez}
\address{Departamento de An\'alisis Matem\'atico, Universidad de M\'alaga, Campus de
Teatinos, 29071 M\'alaga, Spain} \email{japelaez@uma.es}

\author{Elena de la Rosa}
\address{Departamento de An\'alisis Matem\'atico, Universidad de M\'alaga, Campus de
Teatinos, 29071 M\'alaga, Spain} 
\email{elena.rosa@uma.es}

\thanks{This research was supported in part by Ministerio de Econom\'{\i}a y Competitividad, Spain, projects
PGC2018-096166-B-100 and MTM2017-90584-REDT; La Junta de Andaluc{\'i}a,
projects FQM210 S and UMA18-FEDERJA-002.}

%\subjclass[47G10, 30H10, 30H20]{47G10, 30H10, 30H20}

\maketitle

\begin{abstract}
 We consider 
 the Hilbert-type operator defined by
 $$
  H_{\omega}(f)(z)=\int_0^1 f(t)\left(\frac{1}{z}\int_0^z B^{\omega}_t(u)\,du\right)\,\omega(t)dt,$$
  where  $\{B^{\omega}_\zeta\}_{\zeta\in\mathbb{D}}$ are the  reproducing kernels of the Bergman space $A^2_\omega$   induced by a radial weight $\omega$
in the unit disc $\mathbb{D}$.
We prove that $H_{\omega}$ is bounded from $H^\infty$ to  the  Bloch space if and only if
  $\omega$ belongs to the class $\widehat{\mathcal{D}}$, which consists of radial weights $\omega$  satisfying the doubling condition  
  $\sup_{0\le r<1} \frac{\int_r^1 \omega(s)\,ds}{\int_{\frac{1+r}{2}}^1\omega(s)\,ds}<\infty$.
  Further, we  describe the weights $\omega\in \widehat{\mathcal{D}}$ such that $H_\omega$ is bounded on  the Hardy space $H^1$, and
  we  show that for any $\omega\in \widehat{\mathcal{D}}$ and $p\in (1,\infty)$,  $H_\omega:\,L^p_{[0,1)} \to H^p$  is bounded if and only if 
  the Muckenhoupt type condition 
  \begin{equation*}
     \sup\limits_{0<r<1}\left(1+\int_0^r \frac{1}{\widehat{\omega}(t)^p} dt\right)^{\frac{1}{p}}
     \left(\int_r^1 \omega(t)^{p'}\,dt\right)^{\frac{1}{p'}} <\infty,
    \end{equation*}
    holds. Moreover, we address the analogous question  about  the action of $H_{\omega}$ on weighted Bergman 
spaces $A^p_\nu$.  
  
  \end{abstract}
\section{Introduction}
Let $\H(\D)$  denote the space of analytic functions in the unit disc $\D=\{z\in\C:|z|<1\}$.
For $0<p\le\infty$, the Hardy space $H^p$ consists of $f\in\H(\D)$ for which
    \begin{equation*}\label{normi}
    \|f\|_{H^p}=\sup_{0<r<1}M_p(r,f)<\infty
    \end{equation*}
where
    $$
    M_p(r,f)=\left (\frac{1}{2\pi}\int_0^{2\pi}
    |f(re^{i\theta})|^p\,d\theta\right )^{\frac{1}{p}},\quad 0<p<\infty,
    $$
and 
    $$
    M_\infty(r,f)=\max_{0\le\theta\le2\pi}|f(re^{i\theta})|.
    $$
For a  nonnegative function  $\om  \in L^1_{[0,1)}$, the extension to $\D$, defined by $\om(z)=\om(|z|)$ for all $z\in\D$, is called a radial weight. For $0<p<\infty$ and such an $\omega$, the Lebesgue space $L^p_\om$ consists of complex-valued measurable functions $f$ on $\D$ such that
    $$
    \|f\|_{L^p_\omega}^p=\int_\D|f(z)|^p\omega(z)\,dA(z)<\infty,
    $$
where $dA(z)=\frac{dx\,dy}{\pi}$ is the normalized area measure on $\D$. The corresponding weighted Bergman space is
 $A^p_\om=L^p_\omega\cap\H(\D)$. As usual, we write $A^p_\alpha$ for the classical weighted Bergman spaces induced by the 
standard weights $\om(z)=(\a +1)(1-|z|^2)^\alpha$, $\alpha>-1$.
 Throughout this paper we assume $\widehat{\om}(z)=\int_{|z|}^1\om(s)\,ds>0$ for all $z\in\D$, for otherwise $A^p_\om=\H(\D)$.

For any radial weight $\om$, the norm convergence in $A^2_{\om}$ implies the uniform convergence on compact subsets of $\D$, and hence
the point evaluations $L_z$ are bounded linear functionals on $A^2_\om$. Therefore there exist Bergman reproducing kernels $B^\om_z\in A^2_\om$ such that
$$ L_z(f)=f(z)=\langle f, B_z^{\om}\rangle_{A^2_\om}=\int_{\D}f(\z)\overline{B_z^{\om}(\z)}\om(\z)dA(\z),\quad f \in A^2_{\om}.$$

The kernels $z\mapsto\frac{1}{z}\int_0^z B_t^{\om}(\z)\, d\z$, $z\in\D, t\in [0,1)$, give rise to the
 Hilbert-type operator 
\begin{equation*}\label{eq:i1}
    H_{\omega}(f)(z)=\int_0^1 f(t)\left(\frac{1}{z}\int_0^z B^{\om}_t(\z)d\z\right)\,\om(t)dt.
\end{equation*}
The choice $\om=1$, gives the integral representation
\begin{equation*}\label{eq:hilbert} 
H(f)(z)=\int_{0}^1 \frac{f(t)}{1-tz}\,dt 
\end{equation*} 
 of the classical Hilbert operator \cite{DiS},
which  plays a  key role 
in  works concerning the boundedness,  the operator norm  and  the spectrum
of  $H$  \cite{AlMonSa,DiS,Di, DJV}. It is worth mentioning that these papers link the research on $H$ with different topics such as weighted composition operators,
 the Szeg\"{o} projection or Legendre functions of the first kind. We refer to \cite{GaPe2010,GaGiPeSis,PelRathg} for other generalizations of the classical Hilbert operator.

The primary purpose of this paper is to study the operator $H_{\omega}$ acting on Hardy and weighted Bergman spaces. 
 $H_{\omega}$ is an integral operator induced by the kernel $K^{\om}_t(z)=\frac{1}{z}\int_0^z B_t^{\om}(u)\, du$, so a meaningful knowledge
of these kernels would be helpful to understand the behavior of the operator. However, 
 the  lack of explicit expressions for
 $K^{\om}_t(z)$ makes difficult dealing with them, in fact this is one of the main obstacles  throughout this work.
It is well-known that   $B^\om_z(\z)=\sum \overline{e_n(z)}e_n(\z)$ for each orthonormal basis $\{e_n\}$ of $A^2_\om$, and therefore 
\begin{equation}\label{eq:B}
B^\om_z(\z)=\sum_{n=0}^\infty\frac{\left(\overline{z}\z\right)^n}{2\om_{2n+1}}, \quad z,\z\in \D,
\end{equation}
 using the basis
 induced by the normalized monomials. Here $\om_{2n+1}$ are the odd moments of $\om$, and in general from now on we write $\om_x=\int_0^1r^x\om(r)\,dr$ for all $x\ge0$.

As for Hardy spaces, the classical Hilbert operator $H$ is not bounded on $H^\infty$. 
This is a general phenomenon rather than a particular case. 
 Indeed, if $\om$ is a radial weight and $x\in (0,1)$
\begin{align*}
  H_{\omega}(1)(x)
  =\sum\limits_{n=0}^{\infty}\frac{\om_n}{2\om_{2n+1}(n+1)}x^n
  &\geq \frac{1}{2x}\sum\limits_{n=0}^{\infty}\frac{x^{n+1}}{n+1}=\frac{1}{2x}\log\left(\frac{1}{1-x}\right), 
\end{align*}
so  $H_{\omega}$ is not bounded on $H^\infty$.
However, the classical Hilbert operator 
 is bounded from $H^\infty$ to
 the Bloch space $\mathcal{B}$, which consists of the functions $f\in \H(\D)$ such that 
 $\| f\|_{\mathcal{B}}=|f(0)|+\sup_{z\in\D}(1-|z|^2)|f'(z)|<\infty$ \cite{Zhu}. Our first theorem
 describes the radial 
 weights such that $H_\om:\, H^\infty\to \mathcal{B}$ is bounded. 
We need some notation to state this result. A radial weight $\om$ belongs to the class~$\DD$ if the tail integral $\widehat{\om}$ satisfies the doubling property $\widehat{\om}(r)\le C\widehat{\om}(\frac{1+r}{2})$ for some constant $C=C(\om)\ge1$ and for all $0\le r<1$.
\begin{theorem}\label{th:infty}
Let $\om$ be a radial weight. Then,  $H_{\omega}: H^{\infty}\to \B$ is bounded if and only if  $\om\in \DD$.
\end{theorem}

The class~$\DD$ also arises in other natural questions on operator theory on spaces of analytic functions. For instance, 
it describes the radial weights $\om$ such that the Littlewood-Paley inequality 
\begin{equation*}
  |f(0)|^p+\int_{\D}|f'(z)|^p(1-|z|)^p\om(z)\,dA(z)\le C\|f\|^p_{A^p_\om},\quad f\in\H(\D),
	\end{equation*} 
 holds \cite[Theorem~6]{PR19}, or the radial weights such that the orthogonal Bergman projection 
 \begin{equation*}
    P_\om(f)(z)=\int_{\D}f(\z) \overline{B^\om_{z}(\z)}\,\om(\z)dA(\z),
    \end{equation*}
is bounded from $L^\infty$ to $\mathcal{B}$ \cite[Theorem~1]{PR19}.

 Next, let us recall that the classical Hilbert operator is not bounded on $H^1$  \cite{DiS}. Regarding the action of $H_\om$ on $H^1$,  
 we have  
characterized  the upper doubling weights $\om$
such that $H_\om$ is bounded on $H^1$.

\begin{theorem}
\label{th:h1h1}
Let $\om  \in \DD$. Then, the  following statements sare equivalent:
\begin{itemize}
    \item[(i)] $H_{\om }: H^1\to H^1$ is bounded;

    \item[(ii)] The measure $\mu_\om(z)= \om(z)\left(1+\int_0^{|z|} \frac{ds}{\omg(s)}\right)\,\chi_{[0,1)}(z)\,  dA(z)$ is a $1$-Carleson measure for $H^1$;
    \item[(iii)]  $\om$ satisfies the condition 
    \begin{equation}
        \label{condicion H1 H1}
        \sup\limits_{a \in [0,1)}
        \frac{1}{1-a}\int_a^1 \om(t)\left(1+\int_0^t \frac{ds}{\omg(s)}\right)\,dt<\infty.
    \end{equation}    
\end{itemize}
\end{theorem}

For a given Banach space (or a complete metric
space) $X$ of analytic functions on $\D$, a positive Borel measure
$\mu$ on $\D$ is called a $q$-Carleson measure for $X$
 if the identity 
operator $I_d:\, X\to L^q(\mu)$\index{$I_d$} is bounded. 
Carleson provided a geometric description of $p$-Carleson mesures for Hardy spaces $H^p$, $0<p<\infty$, en route of his celebrated Corona theorem \cite{CarlesonL58, CarlesonL62}
or \cite[Chapter~9]{Duren}.

The condition \eqref{condicion H1 H1} implies a restriction in the growth/decay of  $\om$, for instance
the standard weight $(1-|z|^2)^\alpha$  satisfies  \eqref{condicion H1 H1} if and only if $\alpha>0$.
There are two key steps in the proof of Theorem~\ref{th:h1h1}. The first one consists in obtaining estimates for the integral means  
$M_1(\rho, K^\om_t)$, see Lemma~\ref{triangulo} below. 
The second one will be  used 
 repeatedly throughout  this paper,  and
describes the growth of  the kernel $K^\om_s$ in the interval $[0,1)$, 
$$ K^\om_s(t) \asymp 1+\int_0^{st} \frac{1}{\omg(x)( 1-x)} dx, \quad 0\le t, s<1.$$

Next, we will study the operator $H_\om$ acting on Hardy spaces $H^p$, $1<p<\infty$. In order to state our result concerning
this question, we define 
 the Dirichlet-type space $D^p_{p-1}$
of $f\in\H(\D)$ such that
$\| f\|^p_{D^p_{p-1}}=|f(0)|^p+\| f'\|^p_{A^p_{p-1}}<\infty,$
 and the Hardy-Littlewood space $HL(p)$ which consists of the $f(z)=\sum\limits_{n=0}^{\infty} \widehat{f}(n) z^n\in \H(\D)$ such that
$ \nm{f}^p_{HL(p)}=\sum\limits_{n=0}^{\infty} \abs{\widehat{f}(n)}^p (n+1)^{p-2}<\infty.$
These spaces satisfy the well-known inclusions \cite{Duren,Flett,LP}         
\begin{equation}\label{eq:HpDpp<2}
D^p_{p-1}\subset H^p\subset HL(p),\quad 0<p\le 2,
\end{equation}
and
\begin{equation}\label{eq:HpDpp>2}
HL(p)\subset H^p\subset D^p_{p-1} ,\quad 2\le p<\infty,
\end{equation}
and  are closely related to Hardy spaces. In fact, 
a univalent function $f$ belongs to 
 $H^p$ if and only if $ f\in D^p_{p-1}$ \cite{BGP}, and  $$f\in H^p\Leftrightarrow f \in  HL(p) \Leftrightarrow f \in D^p_{p-1}, \quad 1\le p<\infty, $$
 for 
any $f \in \H(\D)$  whose sequence of Taylor coefficients  is positive and decreasing
\cite[Proof of Theorem~A]{Pavdec}.  

For a radial weight $\om$ and $0<p<\infty$, let $L^p_{\om, [0,1)}$ be the Lebesgue space of measurable functions such that 
$\|f\|^p_{L^p_{\om, [0,1)}}=\int_0^1 |f(t)|^p\om(t)\,dt<\infty$. If $\om=1$ we simply write $L^p_{[0,1)}$ instead  of $L^p_{[0,1), 1}$.
Our next result 
describes the weights $\om \in \DD$ such that $H_{\omega}:L^p_{[0,1)} \to X$ is bounded,
where $X= H^p$, $D^p_{p-1}$ or $HL(p)$,
and consequently bearing in mind  the  estimate \cite[p. 127]{Pomm}
\begin{equation}\label{eq:Hprestriction}
\int_0^s M^p_\infty(t,f)\,dt \le \pi M^p_p(s,f),\quad 0<p<\infty,\quad f\in \H(\D),
\end{equation}
we obtain a sharp sufficient condition so that $H_\om$ is bounded on $H^p$.

\begin{theorem}\label{th:lphp}
Let  $\om \in \DD$ and $1<p<\infty$. Then, the following statements are equivalent:
\begin{itemize}
    \item[(i)] $H_{\omega}:L^p_{[0,1)} \to H^p$ is bounded;
     \item[(ii)] $H_{\omega}:L^p_{[0,1)} \to D^p_{p-1}$ is bounded;
     \item[(iii)] $H_{\omega}:L^p_{[0,1)} \to HL(p)$ is bounded;
   %  \item[(ii)] $\widetilde{H}^{\om}:L^p_{[0,1)} \to H^p$ is bounded;
    \item[(iv)]  $\om$ satisfies the condition
    \begin{equation}\label{eq:i1}
     \sup\limits_{0<r<1}\left(1+\int_0^r \frac{1}{\omg(t)^p} dt\right)^{\frac{1}{p}}
     \left(\int_r^1 \om(t)^{p'}\,dt\right)^{\frac{1}{p'}} <\infty.  
    \end{equation}
    \end{itemize}
In particular if  \eqref{eq:i1} holds, $H_{\omega}: H^p \to H^p$ is bounded.
\end{theorem}

In the proof of the equivalences (i)$\Leftrightarrow$(ii)$\Leftrightarrow$(iii) in Theorem~\ref{th:lphp}, we employ a decomposition norm
for the space $D^p_{p-1}$  \cite[Theorem~2.1]{MatPav} to show that 
$$  \nm{H_{\omega} (f)}_{H^p}\asymp \nm{ H_{\omega} (f)}_{D^p_{p-1}}\asymp \nm{H_{\omega} (f)}_{HL(p)}$$ 
for any non-negative $f\in L^1_{\om,[0,1)}$.
The proof of the equivalence  (i)$\Leftrightarrow$(iv) in Theorem~\ref{th:lphp} is quite similar to the proof 
of the next result, which deals with the action of $H_\om$ on weighted Bergman spaces $A^p_\nu$.

\begin{theorem}
\label{acotacion en espacios de Bergman}
Let $1<p<\infty$, $\om\in \DD$ and $\nu$ a radial weight on $\D$. Then,  $H_{\omega}:L^p_{\nug,[0,1)} \to A^p_\nu$ is bounded if and only if
    \begin{equation}
    \label{condicion 1}
     \sup\limits_{0<r<1}\left(1+\int_0^r \frac{\nug(t)}{\omg(t)^p} dt\right)^{\frac{1}{p}}
     \left(\int_r^1 \left(\frac{\om(t)}{\nug(t)^{\frac{1}{p}}}\right)^{p'}\,dt\right)^{\frac{1}{p'}} <\infty.  
    \end{equation}
    In particular if  \eqref{condicion 1} holds, $H_{\omega}:A^p_\nu  \to A^p_\nu$  is bounded.
\end{theorem}

The last assertion in the statement of Theorem~\ref{acotacion en espacios de Bergman} follows 
from the equivalence between (i) and (ii) and the inequality
\begin{equation}\label{eq:lAplpnu}
\int_0^1 |f(t)|^p\widehat{\nu}(t)\,dt \le \pi \int_0^1 M_\infty(t,f)^p\widehat{\nu}(t)\,dt \le \frac{\pi}{2}\| f\|^p_{A^p_\nu}, \quad f\in \H(\D),
\end{equation}
which can be obtained
integrating \eqref{eq:Hprestriction}.

It is worth mentioning that Theorem~\ref{acotacion en espacios de Bergman}  generalizes and improves \cite[Theorem~2]{PelRathg},
because 
we do not need to assume any integrability condition
on the weights  similar to \cite[(1.2)]{PelRathg} neither assuming that  $\nu$ is a regular weight \cite[(2.1)]{PelRathg}.
The proof of (ii)$\Rightarrow$(i)  in Theorem~\ref{acotacion en espacios de Bergman} draws on the inequality
\begin{equation}\label{eq:i2}
 \nm{f}_{A^p_{\nu}}^p\lesssim \abs{f(0)}^p\nug(0)+\int_0^1M_q^p(t,f')(1-t)^{p\left(1-\frac{1}{q}\right)}\nug(t)dt, \quad f\in \H(\D),
 \end{equation}
which is valid for any $1<q<p<\infty$ and  any radial weight $\nu$. Indeed, \eqref{eq:i2} allows both: to get rid of the integral 
in the expression of $K^{\om}_t$ and then use estimates for the Bergman reproducing kernels previously obtained in \cite[Theorem~1]{PelRatproj}.
This approach leads to a couple of classical Hardy inequalities, one of them follows from \eqref{condicion 1}
and the other one holds whenever
\begin{equation*}
    \sup\limits_{0<r<1}(1-r)^{\frac{1}{p}}\nug(r)^{\frac{1}{p}}\left(\int_0^r \left(\frac{\om(t)}{\nug(t)^{\frac{1}{p}}}\right)^{p'}\frac{1}{\omg(t)^{p'}}dt \right)^{\frac{1}{p'}}<\infty,
\end{equation*}
which is a condition that can be deduced from \eqref{condicion 1}.

We also point out that the statements  of Theorems~\ref{th:infty}-\ref{acotacion en espacios de Bergman} remain true if the
Hilbert type operator $H_\om$ is replaced by its
 sublinear version
\begin{equation*}
    \widetilde{H_{\omega}}(f)(z)=\int_0^1 \abs{f(t)}\left(\frac{1}{z}\int_0^z B^{\om}_t(\z)d\z\right)\om(t)dt.
\end{equation*}

Finally, in Section \ref{sec:fc}, we highlight a connection between the actions of the Hilbert-type operator $H_\om$ 
and the Bergman projection $P_\om$ on Lebesgue spaces.
The rest  of the paper is organized as follows. In Section~\ref{sec:estimates} we will deal with 
pointwise and norm estimates of the kernels. Theorem~\ref{acotacion en espacios de Bergman} is proved in Section~\ref{sec:Bergman}
and Section~\ref{sec:Hardy} is devoted to proving Theorems~\ref{th:infty}, \ref{th:h1h1} and \ref{th:lphp}.

The letter $C=C(\cdot)$ will denote an absolute constant whose value depends on the parameters indicated
in the parenthesis, and may change from one occurrence to another.
We will use the notation $a\lesssim b$ if there exists a constant
$C=C(\cdot)>0$ such that $a\le Cb$, and $a\gtrsim b$ is understood
in an analogous manner. In particular, if $a\lesssim b$ and
$a\gtrsim b$, then we write $a\asymp b$ and say that $a$ and $b$ are comparable.

 This notation has already been used above in the introduction.
 
\section{Pointwise and norm estimates of  kernels}\label{sec:estimates}
We will repeatedly use several characterizations 
of weights in the class $\DD$. They are gathered in the following lemma,
see \cite[Lemma~2.1]{PelSum14} for a proof.

\begin{letterlemma}
\label{caract. pesos doblantes}
Let $\om$ be a radial weight on $\D$. Then, the following statements are equivalent:
\begin{itemize}
    \item[(i)] $\om \in \DD$;
    \item[(ii)] There exist $C=C(\om)\geq 1$ and $\b_0=\b_0(\om)>0$ such that
    $$ \omg(r)\leq C \left(\frac{1-r}{1-t}\right)^{\b}\omg(t), \quad 0\leq r\leq t<1;$$
    for all $\b\geq \b_0$.
    \item[(iii)]
    $$ \int_0^1 s^x \om (s) ds\asymp \omg\left(1-\frac{1}{x}\right),\quad x \in [1,\infty);$$
    \item[(iv)] There exists $C(\om)>0$ such that $\om_n\le C \om_{2n}$, for any $n\in \N$.
\end{itemize}
\end{letterlemma}

Write 
$K^\om_{\z}(z)=\frac{1}{z}\int_0^z B^\om_{\z}(u)\,du\, \text{and}\,
G^\om_{\z}(z) =\frac{d}{dz}K^\om_{\z}(z),\, z,\z\in\D. $ It follows from \eqref{eq:B} that 
$G^{\om}_{\z}(z)=
\bar{\z} 
\sum_{k=0}^\infty 
\frac{z^k\bar{\z}^k}{2\om_{2k+3}}\frac{k+1}{k+2}$, so a slight modification in
 the proof of 
\cite[Theorem 1]{PelRatproj} gives the next result.

\begin{letterlemma}\label{le:Gt}
Let $\om\in\DD$ and $0<q<\infty$. Then, 
\begin{equation*}\label{eq:Gtintegralmeans}
M_q^q(r,G^{\om}_\z)\asymp |\z|^q M_q^q (r, B_\z^{\om})\asymp |\z|^q+|\z|^q\int_0^{r|\z|}\frac{dx}{\omg(x)^q(1-x)^q}, \quad 0\le r<1, \quad \z\in\D.
\end{equation*}
\end{letterlemma}

\begin{lemma}
\label{Estimaciones integral real nucleos}
Let  $\om \in \DD$. Then, 
\begin{equation}\label{eq:radialkernel}
 K^\om_s(t)
 \asymp 1+\int_0^{st} \frac{1}{\omg(x)( 1-x)} dx, \quad 0\le t, s<1.
 \end{equation}
\end{lemma}
\begin{proof} 
By using \eqref{eq:B}, we get
$  K^\om_s(t)=
\frac{1}{2\om_1} +\frac{1}{t} \int_0^t \sum\limits_{n=1}^{\infty} \frac{s^n x^n}{2\om_{2n+1}}\,dx,$
so it is enough to estimate $\int_0^t \sum\limits_{n=1}^{\infty} \frac{s^n x^n}{2\om_{2n+1}}dx$.

Assume that $st\in [1/2,1)$ and choose $N^*\in\N$, $N^*\geq 2$, such that $1-\frac{1}{N^*}\leq st< 1-\frac{1}{N^{*}+1}$. %using that $\om \in \DD$ and
 Lemma \ref{caract. pesos doblantes} yields
\begin{equation*}\begin{split}
 & \int_0^t \sum\limits_{n=1}^{\infty} \frac{s^n x^n}{2\om_{2n+1}}dx  \geq \int_0^t \sum\limits_{n=1}^{N^*} \frac{s^n x^n}{2\om_{2n+1}}dx
  \asymp t  \sum\limits_{n=1}^{N^*} \frac{s^n t^n}{(n+1)\omg\left( 1-\frac{1}{n+1}\right)}
\\ &\gtrsim t  \sum\limits_{n=1}^{N^*}\int_n^{n+1} \frac{1}{x\omg\left( 1-\frac{1}{x}\right)} dx
   =t\int_0^{1-\frac{1}{N^*+1}}\frac{1}{\omg(r)(1-r)}dr
   \geq t\int_0^{st}\frac{1}{\omg(r)(1-r)}\,dr.
  \end{split}\end{equation*}
  
 If $st\in[0,\frac{1}{2})$
  \begin{align*}
  \int_0^t \sum\limits_{n=1}^{\infty} \frac{s^n x^n}{2\om_{2n+1}}dx\geq \int_0^t \frac{s x}{2\om_{3}}dx \asymp t\,(st)
  \gtrsim t \int_0^{st} \frac{1}{\omg(x)(1-x)}dx,
  \end{align*}
so
  \begin{equation*}
 \int_0^t B_s^{\om} (x) dx \gtrsim t+t\int_0^{st} \frac{1}{\omg(x)(1-x)} dx, \quad 0\le s,t<1.
   \end{equation*}

  On the other hand, to obtain the upper bound in \eqref{eq:radialkernel}, we will estimate both terms in the   sum 
  \begin{align*}
    \int_0^t \sum\limits_{n=1}^{N^*} \frac{s^n x^n}{2\om_{2n+1}}dx + \int_0^t \sum\limits_{n=N^*+1}^{\infty} \frac{s^n x^n}{2\om_{2n+1}}dx .
  \end{align*}
 Lemma \ref{caract. pesos doblantes} also gives
\begin{align*}
  \int_0^t \sum\limits_{n=1}^{N^*} \frac{s^n x^n}{2\om_{2n+1}}\,dx &\asymp \sum\limits_{n=1}^{N^*} \frac{s^n t^{n+1}}{(n+1)\omg\left(1-\frac{1}{n}\right)}
  \\&
  \lesssim t\sum\limits_{n=1}^{N^*} \int_n^{n+1}\frac{1}{x\omg\left(1-\frac{1}{x}\right)} \, dx
  \\   & =  t\int_0^{1-\frac{1}{N^*+1}}\frac{1}{\omg(x)(1-x)}dx.
\end{align*}

If $0\le st\leq \frac{1}{2}$, ($N^*=1$), then

\begin{equation}\label{eq:j2}
\int_0^t \sum\limits_{n=1}^{N^*} \frac{s^n x^n}{2\om_{2n+1}}\,dx\lesssim t\int_0^{\frac{1}{2}}\frac{1}{\omg(x)(1-x)}dx\lesssim \frac{t}{\omg\left(\frac{1}{2}\right)}\lesssim t.
\end{equation}

If $\frac{1}{2}\le st <1$, by using Lemma~\ref{caract. pesos doblantes}
 \begin{align*}
& t\int_0^{1-\frac{1}{N^*+1}}\frac{1}{\omg(x)(1-x)}dx
\\ &\leq  
   t\int_0^{st}\frac{1}{\omg(x)(1-x)}dx+
 t\int_{1-\frac{1}{N^*}}^{1-\frac{1}{N^*+1}}\frac{1}{(1-x)\omg\left(x\right)} dx 
  \\ &  \le   t\int_0^{st}\frac{1}{\omg(x)(1-x)}dx+\frac{t\log 2}{\omg\left(1-\frac{1}{N^*+1}\right)} 
\\ &
\lesssim t\int_0^{st}\frac{1}{\omg(x)(1-x)}dx+ \frac{t\log 2}{\omg(2st-1)}
\\ & \lesssim t\int_0^{st}\frac{1}{\omg(x)(1-x)}dx+t\int_{2st-1}^{st} \frac{1}{\omg(x)(1-x)}dx \lesssim t\int_0^{st}\frac{1}{\omg(x)(1-x)}dx,
\end{align*}
which together with \eqref{eq:j2} gives
 
\begin{equation}\label{eq:j3}
 \int_0^t \sum\limits_{n=1}^{N^*} \frac{s^n x^n}{2\om_{2n+1}}dx \lesssim t+ t\int_0^{st}\frac{1}{\omg(x)(1-x)}dx, \quad 0\le s,t <1.
 \end{equation}

Moreover,   by Lemma \ref{caract. pesos doblantes}(ii) there is $\b>1$ such that $h(x)=\frac{\omg(x)}{(1-x)^\b}$ is essentially increasing, so
if  $\frac{1}{2}\le st<1$
\begin{align*}
&    \int_0^t \sum\limits_{n=N^*+1}^{\infty} \frac{s^n x^n}{2\om_{2n+1}}dx \asymp \sum\limits_{n=N^*+1}^{\infty} \frac{s^n t^{n+1}}{(n+1)\omg\left(1-\frac{1}{n}\right)}
   \lesssim t \frac{(1-st)^{\b}}{\omg(st)}\sum\limits_{n=N^*+1}^{\infty} (st)^{n}(n+1)^{\b -1} \\
    &\lesssim t \frac{(1-st)^{\b}}{\omg(st)}(N^* +1)^{\b} 
     \lesssim \frac{t}{\omg(st)}
   \asymp t \int\limits_{2st-1}^{st}\frac{1}{\omg(x)(1-x)}dx\leq  t \int\limits_{0}^{st}\frac{1}{\omg(x)(1-x)} dx,
    \end{align*}
and  if 
 $st\leq \frac{1}{2}$
$$  \int_0^t \sum\limits_{n=N^*+1}^{\infty} \frac{s^n x^n}{2\om_{2n+1}}dx\lesssim\frac{t}{\omg(st)} \leq \frac{t}{\omg(1/2)}. $$
In consequence, 
$$  \int_0^t \sum\limits_{n=N^*+1}^{\infty} \frac{s^n x^n}{2\om_{2n+1}}dx \lesssim t+t \int\limits_{0}^{st}\frac{1}{\omg(x)(1-x)} dx,\quad 0\le s,t<1,$$
which together with \eqref{eq:j3} gives 

\begin{equation*}
   \frac{1}{t}\int_0^t B_s^{\om} (x) dx \lesssim 1+\int_0^{st} \frac{1}{\omg(x)(1-x)} dx, \quad 0\le s, t<1.
   \end{equation*}
This finishes the proof. 
\end{proof}

\begin{lemma}
\label{triangulo}
Let $\om \in \DD$. Then,

$$M_1(\rho, K^\om_t)\asymp 
 1+\int_0^{\r t}\frac{ds}{\omg(s)},\quad 0\le t,\r<1.
$$
\end{lemma}
\begin{proof}
Using that $K^\om_t(z)=\sum_{n=0}^{\infty} \frac{t^nz^n}{2(n+1)\om_{2n+1}}$, 
 the Hardy's inequality \cite[Section~3.6]{Duren} yields
$
%\begin{align*}
M_1(\rho, K^\om_t)
 \gtrsim \sum\limits_{n=0}^{\infty} \frac{t^n\r^n}{2\om_{2n+1}(n+1)^2}.
$
%\end{align*}
So arguing as in the proof of Lemma~\ref{Estimaciones integral real nucleos}, we deduce
$$M_1(\rho, K^\om_t)\gtrsim
 1+\int_0^{\r t}\frac{ds}{\omg(s)},\quad 0\le t,\r<1.
$$

Conversely, bearing in mind \eqref{eq:HpDpp<2} and 
Lemma~\ref{le:Gt}
\begin{equation*}\begin{split}
& M_1(\rho, K^\om_t)  
%&=\| (K^\om_t)_\rho\|_{H^1} 
 \lesssim \| (K^\om_t)_\rho\|_{D^1_0}
 \asymp 1 + \int_0^1 s\rho M_1( s\rho, G^\om_t) ds
 \\ &\asymp  1 + \int_0^1 s\rho t \left(1+ \int_0^{s\r t}\frac{dx}{(1-x)\widehat{\om}(x)} \right)\,ds
 \lesssim  1+\int_0^{\r t}\frac{ds}{\omg(s)}, \quad 0\le t,\rho<1.
\end{split}\end{equation*}
This finishes the proof.
\end{proof}

\section{Hilbert-type operators from Lebesgue to Bergman spaces}\label{sec:Bergman}

 For $p,q, \alpha >0$, let
$H^1(p,q, \alpha)$ denote the space of $f \in \H(\D)$ such that
$$ \nm{f}_{H^1(p,q, \alpha)}=\left(\abs{f(0)}^p +\int_0^1 M_q^p(r,f')(1-r)^{\a}dr\right)^{\frac{1}{p}}<\infty.$$

We begin with obtaining a Littlewood-Paley-type inequality for the $A^p_\nu$ norm which is inherited
from the following one \cite[Corollary~3.1]{MatPav} through integration.

\begin{letterlemma}\label{le:Hpmixto}
Let $1<q<p<\infty$. Then, 
$$\|f\|_{H^p}\lesssim  \nm{f}_{H^1\left (p,q, p\left(1-\frac{1}{q}\right)\right)}, \quad f\in \H(\D).$$

\end{letterlemma}

\begin{lemma}
\label{sustituto lit.Pal}
Let $1<q<p<\infty$ and $\nu$ be a radial weight. Then, 
$$ \nm{f}_{A^p_{\nu}}^p\lesssim \abs{f(0)}^p\nug(0)+\int_0^1M_q^p(t,f')(1-t)^{p\left(1-\frac{1}{q}\right)}\nug(t)dt,  \quad f\in \H(\D). $$ 
\end{lemma}
\begin{proof}
Applying Lemma~\ref{le:Hpmixto} to the dilated functions $f_s(z)=f(sz)$ 
\begin{align*}
   \nm{f}_{A^p_{\nu}}^p&=\int_0^1 M_p^p(s, f) s \nu(s) ds =\int_0^1 \nm{f_s}_{H^p}^p s\nu(s)ds\\
   &\lesssim \abs{f(0)}^p\nug(0)+\int_0^1s\nu(s)\int_0^1(1-r)^{p\left(1-\frac{1}{q}\right)}M_q^p(r, (f_s)')dr \,ds\\
   &=\abs{f(0)}^p\nug(0)+\int_0^1s^{p+1}\nu(s)\left(\int_0^1(1-r)^{p\left(1-\frac{1}{q}\right)}M_q^p(rs, f')dr \right)ds
   \\& = \abs{f(0)}^p\nug(0)+\int_0^1s^p\nu(s)\left(\int_0^s\left(1-\frac{t}{s}\right)^{p\left(1-\frac{1}{q}\right)}M_q^p(t, f')dt \right)ds
   \\&\le\abs{f(0)}^p\nug(0)+\int_0^1\nu(s)\left(\int_0^s (s-t)^{p\left(1-\frac{1}{q}\right)}M_q^p(t, f')dt\right)ds
   \\&=\abs{f(0)}^p\nug(0)+\int_0^1 M_q^p(t, f')\left(\int_t^1 \nu(s)(s-t)^{p\left(1-\frac{1}{q}\right)}ds\right)dt\\&
   \le \abs{f(0)}^p\nug(0)+\int_0^1M_q^p(t,f')(1-t)^{p\left(1-\frac{1}{q}\right)}\nug(t)dt.
\end{align*}
\end{proof}

Now we are ready to prove the main result of this section.
\medskip

\begin{Prf}{\em{Theorem~\ref{acotacion en espacios de Bergman}.}}
Assume that \eqref{condicion 1} holds. Then,
\begin{equation}\label{condicion bien definido}
    \int_0^1 \left(\frac{\om(t)}{\nug(t)^{\frac{1}{p}}}\right)^{p'}\,dt<\infty,   
    \end{equation}
and therefore
\begin{align*}
    \int_0^1 \vert f(t)\vert \left|\frac{1}{z}\int_0^z B_t^{\om}(\z) d\z\right|\om (t) dt&\leq \left(\int_0^1 \vert f(t)\vert^p\nug(t)dt\right)^{\frac{1}{p}}\left(\int_0^1 \left|\frac{1}{z}\int_0^z B_t^{\om}(\z) d\z\right|^{p'}\nug(t)^{-\frac{p'}{p}} \om(t)^{p'}dt\right)^{\frac{1}{p'}}\\
    &\leq \Vert f \Vert_{L^p_{\nug,[0,1)}} B_{\abs{z}^{\frac{1}{2}}}^{\om}(\abs{z}^{\frac{1}{2}}) \left(\int_0^1 \left(\frac{\om(t)}{\nug(t)^{\frac{1}{p}}}\right)^{p'}dt\right)^{\frac{1}{p'}}<\infty,
\end{align*}
so $H_{\omega}(f)(z)$ is well defined for any $z\in\D$.
Next, we will show that $H_{\omega}$ is bounded from $L^p_{\nug,[0,1)}$ to $A^p_{\nu}$.  
Take $q\in (1,p)$, then
by Lemma~\ref{sustituto lit.Pal} 
\begin{align}\label{eq:bergman1}
    \nm{H_{\omega}(f)}_{A^p_{\nu}}^p
    \lesssim \abs{H_{\omega}(f)(0)}^p\nug(0)+\int_0^1M_q^p(t,H_{\omega}(f)')(1-t)^{p\left(1-\frac{1}{q}\right)}\nug(t)\,dt.
\end{align}
By using \eqref{condicion bien definido}
\begin{equation}\begin{split}\label{h1}
   \abs{H_{\omega}(f)(0)}^p\nug(0)&\asymp \nug(0)\left(\int_0^1 \abs{f(t)}\om (t)dt\right)^p
   \\ & \leq \nug(0)\left(\int_0^1\abs{f(t)}^p \nug(t)dt\right)\left(\int_0^1 \left(\frac{\om(t)}{\nug(t)^{\frac{1}{p}}}\right)^{p'}dt\right)^{\frac{p}{p'}}\\
   &\lesssim \nm{f}_{L^p_{\nug,[0,1)}}^p.
\end{split}\end{equation}
Now, Minkowski's inequality in continuous form yields 
\begin{align*}
    M_q(r, H_\omega(f)')&=\left(\frac{1}{2\pi}\int_0^{2\pi}\left|\int_0^1 f(t) G^{\om}(re^{i\theta},t)\om(t) dt\right|^q \,d\theta\right)^{\frac{1}{q}}\\ 
    &\leq \int_0^1 \vert f(t)\vert \om(t)\left(\frac{1}{2\pi}\int_0^{2\pi}\left|G^{\om}(re^{i\theta},t)\right|^q d\theta\right)^{\frac{1}{q}} dt= \int_0^1 \vert f(t)\vert \om(t) M_q (r, G_t^{\om} ) dt.
\end{align*}
Moreover, using that $q\in (1,\infty)$ and  Lemma~\ref{le:Gt} 
$$M_q^q(r,G^{\om}_t)\lesssim  M_q^q (r, B_t^{\om})\asymp 1+\int_0^{rt}\frac{dx}{\omg(x)^q(1-x)^q}\asymp \frac{1}{\omg(rt)^q}\frac{1}{(1-rt)^{q-1}},\quad 0\le t<1, $$
So, joining both inequalities
\begin{align*}
\int_0^1M_q^p(r,H_{\omega}(f)')(1-r)^{p\left(1-\frac{1}{q}\right)}\nug(r)dr
  \lesssim &\int_0^1\left[\int_0^1 \frac{\vert f(t)\vert}{(1-rt)^{1-\frac{1}{q}}}  \frac{\om(t)}{\omg(rt)} dt\right]^p(1-r)^{p\left(1-\frac{1}{q}\right)}\nug(r)dr \\
  \asymp & I+ II,
\end{align*}
where
$$I=\int_0^1\left[\int_0^r \frac{\vert f(t)\vert}{(1-rt)^{1-\frac{1}{q}}}  \frac{\om(t)}{\omg(rt)} dt\right]^p(1-r)^{p\left(1-\frac{1}{q}\right)}\nug(r)dr$$ and $$II=\int_0^1\left[\int_r^1 \frac{\vert f(t)\vert}{(1-rt)^{1-\frac{1}{q}}}  \frac{\om(t)}{\omg(rt)} dt\right]^p(1-r)^{p\left(1-\frac{1}{q}\right)}\nug(r)dr.$$

If $t>r$,  $\displaystyle{ \frac{1}{(1-rt)^{1-\frac{1}{q}}}  \frac{1}{\omg(rt)}\asymp\frac{1}{(1-r)^{1-\frac{1}{q}}}  \frac{1}{\omg(r)}}$,
by  Lemma \ref{caract. pesos doblantes}. Therefore, we need to prove the weighted Hardy's inequality 
\begin{align}\label{eq:bergman2}
    II\asymp\int_0^1\left[\int_r^1 \vert f(t)\vert  \om(t) dt\right]^p\frac{\nug(r)}{\omg(r)^p}dr\lesssim \int_0^1 \vert f(t)\vert^p \nug(t)\,dt,
\end{align}
which holds if and only if
(see \cite[Theorem 2]{Muckenhoupt1972} or \cite[Theorem~3.2]{GoLa})
\begin{equation*}
 \sup\limits_{0<r<1}\left(\int_0^r \frac{\nug(t)}{\omg(t)^p} dt\right)^{\frac{1}{p}}\left(\int_r^1 \left(\frac{\om(t)}{\nug(t)^{\frac{1}{p}}}\right)^{p'}\,dt\right)^{\frac{1}{p'}}<\infty.   
\end{equation*}
That is, \eqref{eq:bergman2} follows from (ii).

Now, we will deal with $I$.
If $t<r$, by using Lemma \ref{caract. pesos doblantes}, $\displaystyle{ \frac{1}{(1-rt)^{1-\frac{1}{q}}}  \frac{1}{\omg(rt)}\asymp\frac{1}{(1-t)^{1-\frac{1}{q}}}  \frac{1}{\omg(t)}}$. So, the next step is to show  the weighted Hardy's inequality
\begin{align*}
    I\asymp\int_0^1\left[\int_0^r \frac{\vert f(t)\vert}{(1-t)^{1-\frac{1}{q}}}  \frac{\om(t)}{\omg(t) }dt\right]^p(1-r)^{p\left(1-\frac{1}{q}\right)}\nug(r)dr\lesssim \int_0^1 \vert f(t)\vert^p \nug(t)dt,
\end{align*}
 which holds  if and only if (see \cite[Theorem 1]{Muckenhoupt1972} or \cite[Theorem~3.2]{GoLa})
\begin{align*}
 \sup\limits_{0<r<1} \left( \int_r^1 U^p(t) dt \right)^{\frac{1}{p}}\left(\int_0^r V^{-p'}(t)dt\right)^{\frac{1}{p'}}<\infty,
\end{align*}
where $$U^p(t)=\left\{\begin{array}{cc}
   \nug(t)(1-t)^{p\left(1-\frac{1}{q}\right)} &  \text{ if }0<t<1 \\
    0 & \text{ otherwise}
\end{array}
\right.,\quad V^p(t)=\left\{\begin{array}{cc} \nug(t)(1-t)^{p\left(1-\frac{1}{q}\right)}
   \frac{\omg(t)^p}{\om(t)^p}  &  \text{ if }0<t<1 \\
    0 & \text{ otherwise}
\end{array}
\right. .$$
We observe that
\begin{align*}
   &\left( \int_r^1 U^p(t) dt \right)^{\frac{1}{p}}\left(\int_0^r V^{-p'}(t)dt\right)^{\frac{1}{p'}}\\
   &= \left(\int_r^1 (1-t)^{p\left(1-\frac{1}{q}\right)}\nug(t)dt\right)^{\frac{1}{p}}
   \left(\int_0^r (1-t)^{-p'\left(1-\frac{1}{q}\right)}\frac{ \om (t)^{p'}}{\left(\nug(t)^{1/p}\omg(t)\right)^{p'}}\,dt\right)^{\frac{1}{p'}} 
   \\&\leq (1-r)^{1-\frac{1}{q}}\left(\int_r^1 \nug(t)dt\right)^{\frac{1}{p}}(1-r)^{-(1-\frac{1}{q})}
   \left(\int_0^r \frac{ \om (t)^{p'}}{\left(\nug(t)^{1/p}\omg(t)\right)^{p'}}\,dt\right)^{\frac{1}{p'}}
   \\ & = \left(\int_r^1 \nug(t)dt\right)^{\frac{1}{p}}\left(\int_0^r\frac{ \om (t)^{p'}}{\left(\nug(t)^{1/p}\omg(t)\right)^{p'}}\,dt\right)^{\frac{1}{p'}}\leq (1-r)^{\frac{1}{p}}\nug(r)^{\frac{1}{p}}
 \left(\int_0^r \frac{ \om (t)^{p'}}{\left(\nug(t)^{1/p}\omg(t)\right)^{p'}}
 \,dt \right)^{\frac{1}{p'}}.
\end{align*}
So, bearing in mind \eqref{eq:bergman1}, \eqref{h1}  and \eqref{eq:bergman2},  $H_{\omega}: L^p_{\nug,[0,1)}\to A^p_\nu$ is bounded if \eqref{condicion 1} and the condition
\begin{equation}
    \label{condicion 2}
    \sup\limits_{0<r<1}(1-r)^{\frac{1}{p}}\nug(r)^{\frac{1}{p}}\left(\int_0^r \left(\frac{\om(t)}{\nug(t)^{\frac{1}{p}}}\right)^{p'}\frac{1}{\omg(t)^{p'}}dt \right)^{\frac{1}{p'}}<\infty
\end{equation}
hold. Let us prove that \eqref{condicion 2} follows from \eqref{condicion 1}, this will prove that $H_{\omega}: L^p_{\nug,[0,1)}\to A^p_\nu$ is bounded if \eqref{condicion 1}
holds.

If $1/2\le r<1$, by
 Lemma \ref{caract. pesos doblantes}, 
\begin{align*}
   \int_0^r \frac{\nug(t)}{\omg(t)^p} dt\geq \int_{2r-1}^r \frac{\nug(t)}{\omg(t)^p} dt\geq \frac{\nug(r)(1-r)}{\omg(2r-1)^p} \asymp \frac{\nug(r)(1-r)}{\omg(r)^p},
\end{align*}
and if $0<r<\frac{1}{2}$,
 $ 1\gtrsim \frac{\nug(0)}{\omg(1/2)^p}\gtrsim
\sup_{0<r\le 1/2}  \frac{\nug(r)(1-r)}{\omg(r)^p}$,
\, hence by (\ref{condicion 1})
\begin{equation*}
  \sup\limits_{0<r<1}\frac{\nug(r)^{\frac{1}{p}}(1-r)^{\frac{1}{p}}}{\omg(r)}
  \left(\int_r^1 \left(\frac{\om(t)}{\nug(t)^{\frac{1}{p}}}\right)^{p'}\,dt\right)^{\frac{1}{p'}}<\infty, 
\end{equation*}
that is  
\begin{equation}
\label{cond 1*}
   \int_r^1 \left(\frac{\om(t)}{\nug(t)^{\frac{1}{p}}}\right)^{p'}\,dt\lesssim \left(\frac{\omg(r)}{\nug(r)^{\frac{1}{p}}(1-r)^{\frac{1}{p}}}\right)^{p'}, \quad 0 \le r<1.
\end{equation}
An integration by parts and (\ref{cond 1*}) give 
\begin{align*}
 & \int_0^r \left(\frac{\om(t)}{\nug(t)^{\frac{1}{p}}}\right)^{p'}\frac{1}{\omg(t)^{p'}}dt\\
  &=
  \left[-\omg(t)^{-p'}\int_t^1\left(\frac{\om(s)}{\nug(s)^{\frac{1}{p}}}\right)^{p'} ds\right]_0^r+p'\int_0^r \left(\int_t^1\left(\frac{\om(s)}{\nug(s)^{\frac{1}{p}}}\right)^{p'} ds\right)\omg(t)^{-p'-1}\om(t)dt,
 \\ &\lesssim \frac{1}{\omg(0)^{p'}}\int_0^1\left(\frac{\om(s)}{\nug(s)^{\frac{1}{p}}}\right)^{p'} ds +p'\int_0^r \frac{\om(t)}{\omg(t)}\frac{dt}{((1-t)\nug (t))^{p'/p}}\\ 
 &\leq \frac{1}{\omg(0)^{p'}}\int_0^1\left(\frac{\om(s)}{\nug(s)^{\frac{1}{p}}}\right)^{p'} ds +p'\frac{1}{\nug(r)^{p'/p}}\int_0^r \frac{\om(t)}{\omg(t)}\frac{dt}{(1-t)^{p'/p}}
 \\ & \lesssim \frac{1}{\nug(r)^{p'/p}(1-r)^{\frac{p'}{p}}},\quad 0\le r<1,
\end{align*}
where  the last inequality follows from \eqref{condicion 1} and
\cite[Lemma~2.3]{PPR18} with $\eta(t)=(1-t)^{\frac{p'}{p}-1}$.

Thus, 
\begin{align*}
  \left( \int_0^r \left(\frac{\om(t)}{\nug(t)^{\frac{1}{p}}}\right)^{p'}\frac{1}{\omg(t)^{p'}}dt\right)^{\frac{1}{p'}}\lesssim \frac{1}{  (1-r)^{\frac{1}{p}}\nug(r)^{\frac{1}{p}}}, 
  \quad 0\le r<1,
\end{align*}
and (\ref{condicion 2}) is satisfied. 

\par Conversely, assume that $H_{\omega}:\,L^p_{\nug,[0,1)} \to A^p_{\nu}$ is bounded. We begin with proving
  (\ref{condicion bien definido}). 
By using
 (\ref{eq:lAplpnu}) and Lemma~\ref{Estimaciones integral real nucleos}, 
\begin{equation*}\begin{split}
 \nm{H_{\omega}}^p\nm{\varphi}^p_{L^p_{\nug,[0,1)}} &\ge \nm{H_{\omega}(\varphi)}_{ A^p_\nu}^p
 \\ &\gtrsim \nm{H_{\omega}(\varphi)}_{L^p_{\nug,[0,1)}}^p 
\\ &= \int_0^1\left[\int_0^1\varphi(s)\om(s)\left( \frac{1}{t}\int_0^t B_s^{\om} (x) dx\right) ds\right]^p \nug(t) dt 
  \\ & \gtrsim \left( \int_0^1\varphi(s)\om(s)ds\right)^p,
\end{split}\end{equation*}
for each non-negative $\varphi\in L^p_{\nug,[0,1)}$. So, 
\begin{equation}
\label{1}
 \int_0^1  |\varphi(s) |\om(s)ds\lesssim\left(\int_0^1 |\varphi(t) |^p\nug(t)dt\right)^{\frac{1}{p}}, \, \varphi \in L^p_{\nug,[0,1)}.
\end{equation} 
Now, for each $n \in \mathbb{N}$ 
define
$\varphi_n(t)=\min\set{n,\, \left(\frac{\om(t)}{\nug(t)}\right)^{\frac{p'}{p}}}.$
A direct calculation shows that $\varphi_n^p\nug\le \varphi_n\om$, which together with
  (\ref{1}) applied to $\varphi_n$ and the monotone convergence theorem, gives that \eqref{condicion bien definido} holds.
In order to show (\ref{condicion 1}), consider the family of test functions $\set{\phi_r}_{0<r<1}$ defined as
$$ \phi_r(t)=\left(\frac{\om(t)}{\nug(t)}\right)^{^{\frac{p'}{p}}}\chi_{[r,1)}(t).$$
The inequality \eqref{eq:lAplpnu} yields
\begin{equation}\label{eq:h3}
 \nm{H_{\omega}(\phi_r)}_{L^p_{\nug,[0,1)}}\leq \frac{\pi}{2}\nm{H_{\omega}(\phi_r)}_{A^p_{\nu}}\leq\frac{\pi}{2}\nm{H_{\omega}}\nm{\phi_r}_{L^p_{\nug,[0,1)}}
\lesssim \left(\int_r^1 \left(\frac{\om(t)}{\nug(t)^{\frac{1}{p}}}\right)^{p'}\,dt\right)^{\frac{1}{p}},
\end{equation}
and 
 Lemma \ref{Estimaciones integral real nucleos} gives
\begin{align*}
 \nm{H_{\omega}(\phi_r)}_{L^p_{\nug,[0,1)}}^p  &
  \geq \int_0^r\left[ \int_r^1
  \left(\frac{\om(s)}{\nug(s)^{\frac{1}{p}}}\right)^{p'}
  \left(\frac{1}{t}\int_0^t B_s^{\om} (x) dx\right)ds\right]^p\nug(t)dt
  \\ &\asymp \int_0^r\left[
  \int_r^1
 \left(\frac{\om(s)}{\nug(s)^{\frac{1}{p}}}\right)^{p'}
 \left(1+\int_0^{st} \frac{1}{\omg(x)(1-x)} dx\right)ds\right]^p
 \nug(t)\,dt.
\end{align*}
If $s>r>t$, Lemma \ref{caract. pesos doblantes} implies
\begin{align*}
  1+\int_0^{st} \frac{1}{\omg(x)(1-x)} dx\geq 1+\int_{2st-1}^{st} \frac{1}{\omg(x)(1-x)} dx
  \gtrsim   \frac{1}{\omg(st)}
  \gtrsim \frac{1}{\omg(t)},\quad  \frac{1}{2}\le st <1,
\end{align*}
and 
\begin{align*}
  1+\int_0^{st} \frac{1}{\omg(x)(1-x)} dx
  \gtrsim \frac{1}{\omg(t)}, \quad 
  \quad 0\le st\leq\frac{1}{2}.
\end{align*}
Then,
$$\nm{H_{\omega}(\phi_r)}_{L^p_{\nug,[0,1)}}^p \gtrsim
 \left(\int_r^1\left(\frac{\om(s)}{\nug(s)^{\frac{1}{p}}}\right)^{p'}\,ds\right)^p
 \left(\int_0^r\frac{\nug(t)}{\omg(t)^p}dt\right), $$
 which together with \eqref{eq:h3} and \eqref{condicion bien definido} yields
 (\ref{condicion 1}). This finishes the proof.

\end{Prf}

\section{Hilbert-type operators on Hardy spaces}\label{sec:Hardy}

In this section we will prove Theorems~\ref{th:infty}, \ref{th:h1h1} and \ref{th:lphp}.

\begin{Prf}{\em{Theorem~\ref{th:infty}}.}
We borrow the argument from the proof of \cite[Theorem~1]{PR19}.
Assume that $\om\in \DD$. Then,

\begin{align*}
    \abs{H_{\omega}(f)(0)}\asymp\left|\int_0^1 f(t)\om (t)dt\right|\leq \left(\int_0^1 \om(t)dt\right)\nm{f}_{H^{\infty}}\lesssim \nm{f}_{H^{\infty}}.
\end{align*}
Moreover,
 $\displaystyle{H_{\omega}(f)'(z)=\int_0^1f(t)G^{\om}(z,t)\om(t)dt}$,
with 
$$G^{\om}(z,t)=\sum\limits_{n=1}^{\infty} \frac{t^n z^{n-1}}{2\om_{2n+1}}\frac{n}{n+1}.$$ 
So, by Lemma~\ref{caract. pesos doblantes}(iv) %yields $\om_n\asymp \om_{2n+1}$
\begin{align*}
    \sup\limits_{z \in \D}(1-\abs{z})\abs{H_{\omega}(f)'(z)}&\leq \nm{f}_{H^{\infty}} \sup\limits_{z \in \D}(1-\abs{z}) \int_0^1\left( \sum\limits_{n=1}^{\infty} \frac{t^n \abs{z}^{n-1}}{2\om_{2n+1}}\frac{n}{n+1}\right) \om(t)dt
    \\& \leq \nm{f}_{H^{\infty}} \sup\limits_{z \in \D}(1-\abs{z})  \sum\limits_{n=1}^{\infty} \frac{ \abs{z}^{n-1}\om_n}{2\om_{2n+1}}.
    \\ &  \lesssim 
     \nm{f}_{H^{\infty}} \sup\limits_{z \in \D}(1-\abs{z})  \sum\limits_{n=1}^{\infty} \abs{z}^{n-1} \lesssim  \nm{f}_{H^{\infty}},
\end{align*}
consequently $H_{\omega}: H^\infty \to \mathcal{B}$ is bounded. 
On the other hand, if $H_{\omega}: H^\infty \to \mathcal{B}$ is bounded, for each $N \in \N$
\begin{align*}
    \infty>\nm{H_{\omega} (1)}_{\B} &\geq \sup\limits_{z \in \D}(1-\abs{z}) 
    \left|\int_0^1 G^{\om} (z,t)\om(t)\,dt\right|
    \\ &\geq \sup\limits_{x \in (0,1)}(1-x) \int_0^1 G^{\om} (x,t)\om(t)dt\\
   & \gtrsim  
    \sup\limits_{x \in (0,1)}(1-x) \sum\limits_{n=1}^{N+1} x^{n-1}\frac{\om_n}{\om_{2n+1}}
    \\&
    \gtrsim \frac{1}{N}\sum\limits_{n=E\left(\frac{3N}{4}\right)}^{N+1}  \frac{\om_n}{\om_{2n+1}} \gtrsim \frac{\om_{N+1}}{\om_{2E\left(\frac{3N}{4}\right)+1}},
\end{align*}
that is 
$ \om_{N+1}\lesssim \om_{2E\left(\frac{3N}{4}\right)+1},$ $N\in\N$.
By iterating this inequality, we get
$\om_{N+1}\lesssim\om_{2(N+1)}$ for any $N\geq 24$, and so $\om \in \DD$  by Lemma~\ref{caract. pesos doblantes}.
\end{Prf}
\medskip

\medskip
\begin{Prf}{\em{Theorem~\ref{th:h1h1}.}}
The equivalence between (ii) and (iii) is a byproduct of \cite[Theorem~9.3]{Duren}.

 By using Lemma \ref{triangulo} and the fact that
$\mu_\om(z)=\om(z)\left(1+\int_0^{ \abs{z}}\frac{ds}{\omg(s)}\right)\chi_{[0,1)}(z)\, dA(z)$
 is a $1$-Carleson measure for $H^1$
\begin{align*}
   \|  H_{\omega}(f)\|_{H^1}
    &\leq  \sup_{0\le r<1}\int_0^1 \abs{f(t)}\om(t) M_1(r, K^\om_t)\,dt
    \\&\lesssim \sup_{0\le r<1}\int_0^1 \abs{f(t)}\om(t)\left(1+\int_0^{r t}\frac{ds}{\omg(s)}\right)dt 
    \\ & \leq
    \int_0^1 \abs{f(t)}\om(t)\left(1+\int_0^{ t}\frac{ds}{\omg(s)}\right)\,dt
    \\&\lesssim \| f\|_{H^1},
\end{align*}
so (i) holds. 

Conversely, assume that $H_\om\, : H^1\to H^1$ is bounded, then for any
$a\in [0,1)$,  consider the test functions $f_a(z)=\frac{1-a^2}{(1-az)^2}$,
with $\nm{f_a}_{H^1}=1$. Then,
 by Féjer-Riesz inequality \cite[Theorem~~3.13]{Duren}  and Lemma \ref{Estimaciones integral real nucleos},
\begin{equation}
    \begin{split}
    \label{2*}
      \infty>C\ge \nm{H_{\omega}(f_a)}_{H^1}&\gtrsim  \nm{H_{\omega}(f_a)}_{L^1_{[0,1)}}\asymp\int_0^1\left(\int_0^1 \frac{1-a^2}{(1-as)^2}\left(\frac{1}{t}\int_0^t B_s^{\om}(x)dx\right)  \om (s)ds\right)dt\\
      &\asymp
      \int_0^1\left(\int_0^1 \frac{1-a^2}{(1-as)^2}\left(1+\int_0^{st} \frac{1}{\omg(x)(1-x)}dx\right) \om (s)ds\right)dt
      \\& \geq \int_0^1\left(\int_a^1 \frac{1-a^2}{(1-as)^2} \om (s)ds\right)dt
      \gtrsim \frac{1}{1-a}\int_a^1 \om (s)ds,
    \end{split}
\end{equation}
and 
\begin{equation*}
    \begin{split}
      \infty>C\ge \nm{H_{\omega}(f_a)}_{H^1} & \gtrsim \int_0^1\left(\int_0^1 \frac{1-a^2}{(1-as)^2}\left(1+\int_0^{st} \frac{1}{\omg(x)(1-x)}dx\right) \om (s)ds\right)dt \\&=
       \int_0^1\frac{1-a^2}{(1-as)^2} \om (s)\left(\int_0^1 \left(1+\int_0^{st} \frac{1}{\omg(x)(1-x)}dx\right)dt\right)ds
      \\&\gtrsim \frac{1}{1-a}\int_a^1 \om (s)\left(\int_0^s \left(1+\int_0^{st} \frac{1}{\omg(x)(1-x)}dx\right)dt\right)\,ds.
    \end{split}
\end{equation*}
Now bearing in mind that $t<s$ and $\om \in \DD$, $\displaystyle{1+\int_0^{st} \frac{1}{\omg(x)(1-x)}dx\gtrsim \frac{1}{\omg(t)}}$, so
\begin{equation}
    \begin{split}
    \label{3*}
  \infty> C\ge   \nm{H_{\omega}(f_a)}_{H^1}
      &\gtrsim
       \frac{1}{1-a}\int_a^1 \om (s)\left(\int_0^s\frac{dt}{\omg(t)}\right)ds.
    \end{split}
\end{equation}
Therefore joining (\ref{2*}) and (\ref{3*}), 
$$ \sup_{a\in [0,1)}\frac{1}{1-a}\int_a^1 w(t)\left(1+\int_0^t \frac{ds}{\omg(s)}\right)\,dt<\infty.$$
This finishes the proof.
\end{Prf}

\medskip
Next, we turn to study the action of the Hilbert-type type operator  $H_\om$ on Hardy spaces $H^p$, $1<p<\infty$. With this aim, we need 
to  introduce some extra notation and prove some preliminary results.

If  $g(z)=\sum\limits_{k=0}^{\infty} \widehat{g}(k) z^k \in \H(\D)$ and $n_1, \,n_2 \in \N\cup\set{0}$, $n_1<n_2$, write
$ S_{n_1,n_2} g(z)=\sum\limits_{k=n_1}^{n_2-1} \widehat{g}(k) z^k$ and 
$\De^n g =\sum\limits_{k \in I(n)} \widehat{g}(k)z^k$, where 
$I(n)=\{k\in\N: 2^n\le k<2^{n+1}\}$,\, $n\in \N\cup\{0\}$.    
In particular, 
$\De^n=\De^n(\frac{1}{1-z}) =\sum\limits_{k \in I(n)}z^k$  has the following property \cite[Lemma~2.7]{CPPR}
\begin{equation}\label{eq:Dn}
 \| \De^n\|_{H^p}\asymp 2^{n(1-1/p)}, \quad n\in \N\cup\{0\},\quad 1<p<\infty.
 \end{equation}

The next result can be found in \cite[Lemma~3.4]{LaNoPa}.
\begin{letterlemma}
\label{lema de los lambda}
Let $1<p<\infty$ and $\la=\set{\la_k}_{k=0}^{\infty}$ be a positive and monotone sequence. Let $ g(z)=\sum\limits_{k=0}^{\infty}b_k z^k$ and $(\la g)(z)=\sum\limits_{k=0}^{\infty}\la_k b_k z^k$.
\begin{itemize}
    \item[(a)] If $\set{\la_k}_{k=0}^{\infty}$ is nondecreasing, there exists a constant $C>0$ such that
    $$ C^{-1}\la_{n_1}\nm{S_{n_1,n_2}g}_{H^p}\leq \nm{S_{n_1,n_2}(\la g)}_{H^p}\leq C \la_{n_2}\nm{S_{n_1,n_2}g}_{H^p}.$$
    \item[(b)] If $\set{\la_k}_{k=0}^{\infty}$ is nonincreasing, there exists a constant $C>0$ such that
    $$ C^{-1}\la_{n_2}\nm{S_{n_1,n_2}g}_{H^p}\leq \nm{S_{n_1,n_2}(\la g)}_{H^p}\leq C \la_{n_1}\nm{S_{n_1,n_2}g}_{H^p}.$$
\end{itemize}

\end{letterlemma}

 Let us recall that
\begin{equation*}
    \widetilde{H_{\omega}}(f)(z)=\int_0^1 \abs{f(t)}\left(\frac{1}{z}\int_0^z B^{\om}_t(\z)d\z\right)\om(t)\,dt.
\end{equation*}
 
\begin{lemma}
\label{equivalencia norma}
 Let $\om \in \DD$ and $1<p<\infty$. Then, 
$$  \nm{\widetilde{H_{\omega}} (f)}_{H^p}\asymp \nm{\widetilde{H_{\omega}} (f)}_{D^p_{p-1}}\asymp \nm{\widetilde{H_{\omega}} (f)}_{HL(p)}, \quad f \in L^1_{\om,[0,1)}.$$ 
\end{lemma}

\begin{proof}
\cite[Theorem~2.1]{MatPav} gives
 
 \begin{align*}
     \nm{\widetilde{H_{\omega}} (f)}_{D^p_{p-1}}^p\asymp\left(\int_0^1 \abs{f(t)}\om(t)dt\right)^p+ \sum\limits_{n=0}^{\infty} 2^{-np}\nm{\De^n (H_{\omega}(f))'}_{H^p}^p.
 \end{align*}
Now, observe that \begin{align*}
  (\widetilde{H_{\omega}} (f))'(z)&=\sum\limits_{n=1}^{\infty}\frac{n}{2(n+1)\om_{2n+1}}\left(\int_0^1 \abs{f(t)}t^n\om(t)dt\right)z^{n-1}
  \\ &=\sum\limits_{n=0}^{\infty}\frac{n+1}{2(n+2)\om_{2n+3}}\left(\int_0^1 \abs{f(t)}t^{n+1}\om(t)dt\right)z^n.
  \end{align*}
  Then, by using Lemma~\ref{lema de los lambda} and \eqref{eq:Dn}
 \begin{align*}
  \nm{\De^n (\widetilde{H_{\omega}}(f))'}_{H^p}&\gtrsim \frac{2^n+1}{2^{n+1}+2}\frac{\nm{\De^n}_{H^p}}{\om_{2^{n+1}+3}}\left(\int_0^1 \abs{f(t)}t^{2^{n+1}+1}\om(t)dt\right)\\
  &\asymp\frac{2^{n(1-1/p)}}{\om_{2^{n+1}+3}}\left(\int_0^1 \abs{f(t)}t^{2^{n+1}+1}\om(t)dt\right),
 \end{align*}
and
\begin{align*}
  \nm{\De^n (\widetilde{H_{\omega}}(f))'}_{H^p}&\lesssim \frac{2^{n+1}+1}{2^{n}+2}\frac{\nm{\De^n}_{H^p}}{\om_{2^{n+2}+3}}\left(\int_0^1 \abs{f(t)}t^{2^{n}+1}\om(t)dt\right)\\
  &\asymp\frac{2^{n(1-1/p)}}{\om_{2^{n+2}+3}}\left(\int_0^1 \abs{f(t)}t^{2^{n}+1}\om(t)dt\right).
 \end{align*}
  By Lemma \ref{caract. pesos doblantes}, $\om_{2^{n+1}+3}\asymp \om_{2^{n+2}+3}\asymp \om_{2^{n+1}} $, so
 \begin{align*}
     \nm{\widetilde{H_{\omega}} (f)}_{D^p_{p-1}}^p&\gtrsim\left(\int_0^1 \abs{f(t)}\om(t)dt\right)^p+ \sum\limits_{n=0}^{\infty} 2^{-n}\frac{1}{ (\om_{2^{n+1}})^p }\left(\int_0^1 \abs{f(t)}t^{2^{n+1}+1}\om(t)dt\right)^p\\
     &\gtrsim \left(\int_0^1 \abs{f(t)}\om(t)dt\right)^p+ \sum\limits_{n=0}^{\infty} 2^{-2n}\frac{1}{ (\om_{2^{n+1}})^p }\sum_{k=2^{n+1}+1}^{2^{n+2}}
     \left(\int_0^1 \abs{f(t)}t^{k}\om(t)dt\right)^p 
     \\&\gtrsim \sum\limits_{k=0}^{\infty}\frac{1}{ (\om_{2k+1})^p(k+1)^2 }\left(\int_0^1 \abs{f(t)}t^{k}\om(t)dt\right)^p\asymp\nm{\widetilde{H_{\omega}} (f)}_{HL(p)}^p
 \end{align*}
 and
 \begin{align*}
     \nm{\widetilde{H_{\omega}} (f)}_{D^p_{p-1}}^p&\lesssim\left(\int_0^1 \abs{f(t)}\om(t)dt\right)^p+ \sum\limits_{n=0}^{\infty} 2^{-n}\frac{1}{ (\om_{2^{n+1}})^p }\left(\int_0^1 \abs{f(t)}t^{2^{n}}\om(t)dt\right)^p
     \\ &\lesssim \sum\limits_{k=0}^{\infty}\frac{1}{ (\om_{2k+1})^p(k+1)^2 }\left(\int_0^1 \abs{f(t)}t^{k}\om(t)dt\right)^p\asymp\nm{\widetilde{H_{\omega}} (f)}_{HL(p)}^p.
     \end{align*}
Bearing in mind \eqref{eq:HpDpp<2} and \eqref{eq:HpDpp>2}, this finishes the proof.
\end{proof}

The following result shows that 
\eqref{eq:Hprestriction} can be extended (with an imprecise constant) to
$H^1\left(p,q,p\left(1-\frac{1}{q}\right)\right)$ spaces.

\begin{lemma}
\label{desigualdad Lp y H1(p,q, p(1-1/q))}
Let  $p\in (0,\infty)$ and $1<q<\infty$. Then,
$$\int_0^1 M^p_\infty(s,f) ds\leq C \| f\|_{H^1\left(p,q,p\left(1-\frac{1}{q}\right)\right)}^p, \quad f\in \H(\D).
$$
\end{lemma}
\begin{proof}
We split the proof according to the values of $p$ and $q$.
If $p<q<\infty$  by applying \cite[Lemma~3.4]{GPP06} and
 \cite[Theorem~1.1]{PavP} with $\om(r)=(1-r)^{-\frac{p}{q}}$,
\begin{align*}
 \int_0^1 M_{\infty}^p(s,f)\, ds &\lesssim \int_0^1 M_q^p\left(\frac{1+s}{2},f\right)(1-s)^{-\frac{p}{q}}ds
 \\ & \lesssim  \int_0^1 M_q^p(r,f)(1-r)^{-\frac{p}{q}}dr \\&
 \lesssim  \left(\abs{f(0)}^p+\int_0^1 M_q^p(r, f')(1-r)^{p\left(1-\frac{1}{q}\right)}dr\right), \quad f\in \H(\D).
\end{align*}
Moreover, if $1<q<p<\infty$,  
\eqref{eq:Hprestriction} and Lemma~\ref{le:Hpmixto} yield,
$$ \int_0^1 M_{\infty}^p(s,f)\, ds \lesssim  \nm{f}_{H^p}^p \lesssim \left(\abs{f(0)}^p+\int_0^1 M_q^p(r, f')(1-r)^{p\left(1-\frac{1}{q}\right)}dr\right), \quad f\in \H(\D). $$
Finally, if $p=q$, by applying \cite[Lemma~4]{GaGiPeSis} with $\a=p-1$, 
$$\int_0^1 M_{\infty}^p(s,f) ds \lesssim \nm{f}_{D^p_{p-1}}^p,
 \quad f\in \H(\D). $$
This finishes the proof.
\end{proof}

Now, we are ready to prove our next result, which includes Theorem~\ref{th:lphp},
and ensures that $H^p$ may be replaced by $HL(p)$ and by any $H^1\left(p,q,p\left(1-\frac{1}{q}\right)\right)$
(in particular by $D^p_{p-1}$) in the stament of Theorem~\ref{th:lphp}.

\begin{theorem}\label{th:hpcuerpo}
Let $1<p<\infty$ and $\om \in \DD$. The following statements are equivalent:
\begin{itemize}
    
    \item[(i)] $H_{\omega}:L^p_{[0,1)} \to H^p$ is bounded;
    \item[(ii)] $H_{\omega}:L^p_{[0,1)} \to H^1\left(p,q,p\left(1-\frac{1}{q}\right)\right)$ is bounded
    for any $1<q<\infty$;
    \item[(iii)]  $H_{\omega}:L^p_{[0,1)} \to H^1\left(p,q,p\left(1-\frac{1}{q}\right)\right)$ is bounded
    for some $1<q<\infty$;
    \item[(iv)] $H_{\omega}:L^p_{[0,1)} \to HL(p)$ is bounded;  
    \item[(v)] $\om$ satisfies the condition
    \begin{equation}\label{eq:Hpcondition}
     \sup\limits_{0<r<1}\left(1+\int_0^r \frac{1}{\omg(t)^p} dt\right)^{\frac{1}{p}}
     \left(\int_r^1 \om(t)^{p'}\,dt\right)^{\frac{1}{p'}} <\infty.  
    \end{equation}
    \end{itemize}
\end{theorem}

\begin{Prf}{\em{Theorem~\ref{th:hpcuerpo}.}}
Firstly, we will prove (i)$\Rightarrow$(v)$\Rightarrow$(ii)$\Rightarrow$(i). 
Assume that (i) holds, then
by Fejer-Riesz inequality
\cite[Theorem 3.13]{Duren}, 
\begin{equation}
\label{eq:Hp1}
\| H_{\omega}(f)\|_{L^p_{[0,1)}}\lesssim \| H_{\omega}(f)\|_{H^p}\lesssim \| f\|_{L^p_{[0,1)}},
\end{equation}
so applying the test functions $\varphi_r(t)=\om(t)^{\frac{p'}{p}}\chi_{[r,1)}(t)$ to \eqref{eq:Hp1}
and  arguing as in the proof of   Theorem~\ref{acotacion en espacios de Bergman} (part (i)$\Rightarrow$(ii)),
(v) follows.

Now, in the proof of Theorem~\ref{acotacion en espacios de Bergman} (part (ii)$\Rightarrow$(i)),
we have also  proved that \eqref{condicion 1} implies
$$\abs{H_{\omega}(f)(0)}^p\nug(0)+\int_0^1M_q^p(t,H_{\omega}(f)')(1-t)^{p\left(1-\frac{1}{q}\right)}\nug(t)\,dt\lesssim \|f\|_{L^p_{\widehat{\nu}}}.$$
So replacing $\widehat{\nu}$ by $1$ in the above inequality and mimicking that proof, one can show that
\eqref{eq:Hpcondition} and the condition
\begin{equation}\label{eq:hp2}
\int_0^r 
  \left(\frac{ \om (t)}{\omg(t) }\right)^{p'}\,dt
  \lesssim \frac{1}{(1-r)^{\frac{p'}{p}}}, \quad 0<r<1,
\end{equation}
imply that (ii) holds. Now observe that 
\eqref{eq:Hpcondition} implies 
\begin{equation}\label{eq:hp3}
\int_r^1 \om(t)^{p'}\,dt
  \lesssim \frac{\widehat{\om}(r)^{p'}}{(1-r)^{\frac{p'}{p}}}, \quad 0<r<1.
\end{equation} 
 An integration by parts, \eqref{eq:hp3}
 and \cite[Lemma~2.3]{PPR18}  with $\eta(t)=(1-t)^{\frac{p}{p'}-1}$,
 gives that \eqref{eq:hp2} holds.
So, (v)$\Rightarrow$(ii).  
Next, (i) follows from (ii)  together with Lemma~\ref{le:Hpmixto}.

Obviously (ii)$\Rightarrow$(iii), and (iii) together with 
 Lemma~\ref{desigualdad Lp y H1(p,q, p(1-1/q))} yields
 $$\| H_{\omega}(f)\|_{L^p_{[0,1)}}\lesssim  \| f\|_{L^p_{[0,1)}},$$
 which implies (v). Therefore, (ii)$\Leftrightarrow$(iii).

It is clear that (iv) holds if and only if $\widetilde{H_{\omega}} :\, L^p_{[0,1)} \to HL(p)$ is bounded. So, 
 (i)$\Rightarrow$(iv) follows from Lemma~\ref{equivalencia norma}. Conversely if (iv) holds,
$\widetilde{H_{\omega}}:\, L^p_{[0,1)} \to H^p$ is bounded by Lemma~\ref{equivalencia norma}, which together with
Fejer-Riesz inequality \cite[Theorem 3.13]{Duren} gives
\begin{equation*}
\|\widetilde{H_{\omega}}(f)\|_{L^p_{[0,1)}}\lesssim \| \widetilde{H_{\omega}}(f)\|_{H^p}\lesssim \| f\|_{L^p_{[0,1)}}.
\end{equation*}
This inequality implies (v).
This finishes the proof.
\end{Prf}

\section{Further comments} \label{sec:fc}
In this section we  show that for any   radial doubling weight $\om$,   the Hilbert-type operator $H_\om$, the Bergman projection $P_\om$
    and its maximal version
    \begin{equation*}
    P^+_\om(f)(z)=\int_{\D}f(\z) \left|\overline{B^\om_{z}(\z)}\right|\,\om(\z)dA(\z),
    \end{equation*}
    are simultaneously bounded when they act on different spaces of functions. To be precise, 
we write $\om\in\Dd$ if there exist constants $K=K(\om)>1$ and $C=C(\om)>1$ such that $\widehat{\om}(r)\ge C\widehat{\om}\left(1-\frac{1-r}{K}\right)$ for all $0\le r<1$, 
and set $\DDD=\DD\cap\Dd$. The weights in the class $\DDD$ are usually called radial doubling weights.
A radial weight $\om$ belongs to the class 
  $\mathcal{M}$ if there exist constants $C=C(\om)>1$ and $K=K(\om)>1$ such that $\om_{x}\ge C\om_{Kx}$ for all $x\ge1$.
 It has recently proved 
that $\Dd$ is a proper subset of $\mathcal{M}$ \cite[Proposition~14]{PR19} but  $\DDD=\DD\cap\mathcal{M}$ \cite[Theorem~3]{PR19}. 
In addition, we observe that $\om_x=x\widehat{\om}_{x-1}, \quad x>0$, so 
$\om\in\DD\Leftrightarrow  \widehat{\om}\in\DD$ by Lemma~\ref{caract. pesos doblantes}(iv).
Bearing in mind this information and the fact that any radial non-increasing weight belongs to $\Dd$, the next result 
follows putting together   Theorem~\ref{acotacion en espacios de Bergman} and   \cite[Theorem~13]{PR19}.

\begin{theorem}\label{th:HomPom}
Let $1<p<\infty$, $\om\in \DD$ and $\nu$ a radial weight. Then, the following statements are equivalent:
\begin{itemize}
\item[(i)] $P_{\om}:L^p_{\nug}\to A^p_{\nug}$ is bounded;
\item[(ii)] $P^+_{\om}:L^p_{\nug}\to L^p_{\nug}$ is bounded;
 \item[(iii)] $\om\in \DDD$, $\nu\in\DD$ and $H_{\omega}:L^p_{\nug,[0,1)} \to A^p_\nu$ is bounded;
    \item[\rm(iv)] $\om\in \DDD$ and $\nu\in\DD$ and
    $$A_p(\om,\nug)= \sup_{0\le r<1}\frac{\widehat{\nu}(r)^{\frac{1}{p}}\left( \int_r^1  \left(\frac{\om(s)}{\nug(s)^{\frac1p}}\right)^{p'
    } \,ds\right)^{\frac{1}{p'}}}{\widehat{\om}(r)}<\infty;$$
    \item[\rm(v)]
    $\om\in \DDD$ and $\nu\in\DD$ and
     $$\sup\limits_{0<r<1}\left(1+\int_0^r \frac{\nug(t)}{\omg(t)^p} dt\right)^{\frac{1}{p}}
     \left(\int_r^1 \left(\frac{\om(t)}{\nug(t)^{\frac{1}{p}}}\right)^{p'}\,dt\right)^{\frac{1}{p'}} <\infty. $$
    \end{itemize}
\end{theorem}

In particular, Theorems~\ref{th:HomPom} says that $P_{\om}:L^p_{\nug}\to A^p_{\nug}$ and $H_{\omega}:L^p_{\nug,[0,1)} \to A^p_\nu$ are simultaneously bounded
whenever $\om\in\DDD$, $\nu\in\DD$ and $1<p<\infty$.

Finally, joining Theorem~\ref{th:lphp} and   \cite[Theorem~13]{PR19}  
it can be obtained an analogous result to Theorem~\ref{th:HomPom} related to the boundedness of the operator $H_{\omega}:L^p_{[0,1)} \to H^p$. Indeed, if 
$\om\in\DDD$ and $1<p<\infty$, then
$P_{\om}:L^p\to A^p$ is bounded if and only if  $H_{\omega}:L^p_{[0,1)} \to H^p$ is bounded.

\end{document}